\documentclass[12pt]{amsart}

\usepackage{amsfonts,amsthm,amsmath,amssymb}
\usepackage{enumerate}
\usepackage[all]{xy}

\newcommand{\bb}[1]{\mathbb{#1}}
\newcommand{\boldalpha}{\mbox{\boldmath$\alpha$\unboldmath}}
\newcommand{\C}{\bb{C}}
\newcommand{\eps}{\epsilon}
\newcommand{\cur}[1]{\mathcal{#1}}
\renewcommand{\epsilon}{\varepsilon}
\renewcommand{\hom}{\mathrm{Hom}}
\newcommand{\ext}{\mathrm{Ext}}
\newcommand{\im}{\mathrm{Im}}
\newcommand{\K}{\widetilde{K}}
\newcommand{\lhra}{\lhook\joinrel\longrightarrow}
\newcommand{\lmt}{\longmapsto}
\newcommand{\lra}{\longrightarrow}
\newcommand{\notensor}[2]{#1#2}
\newcommand{\R}{\mathbb{R}}
\newcommand{\res}[2]{\mathrm{Res}^{#1}_{#2}\,}
\newcommand{\Smash}{\wedge}
\newcommand{\bin}[2]{
\left(
\begin{array}{c}
#1\\
#2
\end{array}
\right) }
\newcommand{\tensor}{\otimes}
\UseComputerModernTips

\newtheorem{lemma}[equation]{Lemma}
\newtheorem{proposition}[equation]{Proposition}
\newtheorem{theorem}[equation]{Theorem}
\theoremstyle{definition}
\newtheorem{definition}[equation]{Definition}
\newtheorem{remark}[equation]{Remark}
\newtheorem{remarks}[equation]{Remarks}
\newtheorem{notation}[equation]{Notation}

\numberwithin{equation}{section}

\begin{document}

\title[Poincar\'{e} duality for $K^*_G(\C P(V))$]{Poincar\'{e} duality for $K$-theory of equivariant complex projective spaces}
\author{J.P.C.~Greenlees}
\address{J.P.C.~Greenlees, Department of Pure Mathematics, Hicks Building, Sheffield S3 7RH. UK.}\email{j.greenlees@sheffield.ac.uk}
\author{G.R.~Williams}
\address{G.R.~Williams, Department of Pure Mathematics, The Open
  University, Walton Hall, Milton Keynes, MK7 6AA. UK.} \email{g.r.williams@open.ac.uk}
\date{\today}
\maketitle

\begin{abstract}
  We make explicit {\em Poincar\'{e} duality} for the equivariant
 $K$-theory of equivariant complex projective spaces. The 
 case of the trivial group provides a new approach to the $K$-theory 
 orientation \cite{ja:shagh}.
\end{abstract}

\section{Introduction}
In \footnote{MSC 2000 57R91, 55N15, 55N91} well behaved cases one expects the cohomology of a finite complex to be
a contravariant functor of its  homology. However,  orientable manifolds
have the special property that the cohomology is covariantly isomorphic to
the homology, and hence in particular the cohomology ring is self-dual.
More precisely, Poincar\'e duality states that taking the cap product
with a fundamental class gives an isomorphism between homology and cohomology
of a manifold.

Classically, an $n$-manifold $M$ is a topological space locally modelled
on $\R^n$, and the fundamental class of $M$ is a homology
class in $H_n(M)$. Equivariantly, it is much less clear how things should
work. If we pick a point $x$ of a smooth $G$-manifold, the tangent space $V_x$
is a representation of the isotropy group $G_x$, and its $G$-orbit is
locally modelled on $G \times_{G_x} V_x$;
 both $G_x$ and  $V_x$ depend on the point $x$. It may happen that
we have a $W$-manifold, in the sense that there
is a single representation $W$ so that $V_x$ is the restriction of $W$ to
$G_x$ for all $x$, but this is very restrictive.
Even if there are fixed points $x$,
the representations $V_x$ at different points need not be equivalent.
It is therefore not clear even in which dimension we should hope to find a
fundamental class. In general one needs complicated apparatus to provide
a suitable context  \cite{CMW}, and ordinary cohomology is especially
complicated.  Fortunately, particular examples can be better behaved.

The purpose of the present
paper is to look at the very concrete example of linear complex projective
spaces: these are not usually $W$-manifolds for any $W$, but we observe that
in equivariant $K$-theory there is a natural choice of fundamental class, and
we make the resulting Poincar\'e duality isomorphism explicit. In the
non-equivariant case this gives an  elementary approach to the classical
$K$-theory fundamental class \cite{ja:shagh}.

\section{Preliminaries}


\subsection{Linear projective spaces.}

Let $V$ be a unitary complex representation of a finite group
$G$. We write $S(V)$ for the unit sphere, $D(V)$ for the unit
disc in $V$, and $S^V$ for the one-point compactification,
$S^V=D(V)/S(V)$. We write $T$ for the circle group
$T=\{\lambda\in\C~|~|\lambda|=1\}$ and $z$ for the natural
representation of $T$.

\begin{definition}
We write $\C P(V)$ for the $G$-space of complex lines in $V$,
so that
$$
\C P(V) \cong S(V\otimes z)/T.
$$
\end{definition}

\subsection{Equivariant stable homotopy theory}

Although our principal results are stated in terms of homology and
 cohomology, we often work in the  equivariant stable homotopy category.
We summarise some standard  results  (see \cite{ja:poeshfcl}, \cite{ll:esht} or \cite[XVI \S5]{jm:ehactdttmorjp} for details). The relevance arises since equivariant homology
and cohomology theories are represented by $G$-spectra in the sense that for based
$G$-spaces $X$,
$$\widetilde{E}_G^*(X) =[X, E]_G^* \mbox{ and } \widetilde{E}^G_*(X)=[S^0,E \Smash X]^G_*,$$
where $E$ is the representing $G$-spectrum of the theory.

\begin{lemma}[Change of groups \mbox{\cite[II.4.3 and II.6.5]{ll:esht}}]\label{lemma:change of group}
Let $H$ be a subgroup of $G$, and suppose that $A$ is an
$H$-spectrum and $B$ is a $G$-spectrum. Then there are natural
isomorphisms
$$
\theta:[A,B]_H\stackrel{\cong}{\lra}[G_+\Smash_H A,B]_G~\mathrm{and}~\phi:[B,A]_H\stackrel{\cong}{\lra}[B,G_+\Smash_H A]_G.
$$
\qed
\end{lemma}

\begin{theorem}[Adams isomorphism \mbox{\cite[II.7.1]{ll:esht}}]
Suppose $B$ is  a $T$-free $(G\times T)$-spectrum. For any  $G$-spectrum $A$
there is a natural isomorphism
$$
[A,\Sigma B/T]_G\cong[ A,B]_{G\times T},
$$
induced by a suitable transfer map.
\qed
\end{theorem}

\subsection{Spanier-Whitehead duality.}
Using function spectra we may define the functional duality
functor $DX=F(X,S^0)$ on $G$-spectra $X$. When restricted to
finite $G$-spectra, the natural map $X \lra D^2X$ is an
equivalence, and one may give a more concrete description: if $X$
is a based $G$-space which embeds in the sphere $S^{1 \oplus V}$, we have
$$\Sigma^VDX\simeq S^{1\oplus V} \setminus X, $$
where we have supressed notation for the suspension spectrum.
The formal properties of the category of $G$-spectra give a useful statement
relating homology and cohomology.

\begin{lemma}[Spanier-Whitehead duality \mbox{\cite[III.2.9]{ll:esht}}]
If $X,Y$ are finite $G$-$CW$-spectra and $E$ is a $G$-spectrum, then
\begin{enumerate}[(i)]
\item there is an isomorphism $SW:E^*_G(X)\stackrel{\cong}{\lra}E_*^G(DX)$;
\item a $G$-map $f:X\lra Y$ gives rise to a commutative diagram
$$
\begin{xy}
\xymatrix{
E^*_G(Y)\ar^-{f^*}[r]\ar_{SW}^\cong[d]&E^*_G(X)\ar^{SW}_\cong[d]\\
E_*^G(DY)\ar^-{(Df)_*}[r]&E_*^G(DX).
}
\end{xy}
$$
\end{enumerate}
\qed
\end{lemma}

\subsection{Equivariant $K$-theory}\label{subsection:K theory}

We are concerned with the equivariant $K$-theory of Atiyah and
Segal \cite{gs:ekt} of finite $G$-CW-complexes, so that
$K^0_G(X)$ is the Grothendieck group of equivariant vector
bundles over $X$, and $K^*_G$ is $R(G)$ in even degrees and zero
in odd degrees. We use the represented extension to arbitrary
spectra: there is a $G$-spectrum $K$ so that
  for  a based $G$-space $X$ we have
$$
\widetilde{K}^0_G(X)=[X,K]_G\mathrm{\ and\ }\widetilde{K}_0^G(X)=[S^0,K\Smash X]_G.
$$

Equivariant $K$-theory has its version of the {\em Thom isomorphism}: if
$E$ is a bundle
over $X$ then we have an isomorphism
$\tau:\widetilde{K}^*_G(X)\stackrel{\cong}{\lra}\widetilde{K}^*_G(X^E)$, 
where $X^E$ denotes the Thom space of $E$.
The isomorphism is made explicit in \cite[\S3]{gs:ekt}, and this permits a definition of the {\em Euler class} $\chi(V)=i^*_V\tau(1)\in K^*_G$, where $i_V$ is the inclusion $S^0\lhra S^V$. In turn, this paves the way for the equivariant {\em Bott periodicity}.

\begin{theorem}[Equivariant Bott periodicity \cite{gs:ekt}]
For a based $G$-space $X$, and a complex representation $V$ of $G$,
multiplying by the Bott class $\tau(1)\in\K^0_G(S^V)$ gives a natural isomorphism
$$
\K^0_G(X)\stackrel{\cong}{\lra}\K^0_G(S^V\Smash X).
$$
Moreover, if $\dim_\C(V)=n$ then
$$
\chi(V)=1-\lambda V+\lambda^2V-\cdots+(-1)^n\lambda^nV\in R(G),
$$
where $\lambda^rV$ denotes the $r^{th}$ exterior power of $V$.
\qed
\end{theorem}

\subsubsection{Restriction in equivariant $K$-theory.} For $H\leq G$,
let $\pi:G/H\lra G/G$ denote projection. It is not hard to verify from the explicit form of the change of groups isomorphisms
that the restriction maps in homology and cohomology are represented in the
following sense.

\begin{lemma}
\label{resrep}
There are commutative diagrams
$$
\begin{xy}
\xymatrix{
\K^0_G(X)\ar@{=}[d]\ar^-{\res{G}{H}}[rr]&&\K^0_H(X)\ar^\theta_\cong[d]\\
\K^0_G(G/G_+\Smash X)\ar^-{(\pi\Smash1)^*}[rr]&&\K^0_G(G/H_+\Smash X)
}
\end{xy}
$$
and
$$
\begin{xy}
\xymatrix{
\K_0^G(X)\ar@{=}[d]\ar^-{\res{G}{H}}[rr]&&\K_0^H(X)\ar^\phi_\cong[d]\\
\K_0^G(G/G_+\Smash
X)\ar^-{(D(\pi)\Smash1)_*}[rr]&&\K_0^G(G/H_+\Smash X). }
\end{xy}
$$
\qed
\end{lemma}

The restriction maps are not, in general, injective. However, one finds that
\begin{equation}\label{equation:injective restriction}
\res{G}{*}:K^0_G(\C P(V))\stackrel{\{\res{G}{H}\}}{\lra}\prod\limits_{\genfrac{}{}{0pt}{}{H\leq G}{H\mathrm{\ cyclic}}}K^0_H(\C P(V))
\end{equation}
and the analogous map in homology {\em are} both injective. This is easily deduced
from the corresponding statement about representation rings. For example, 
it follows from the calculations in Subsection \ref{subsec:KCPV} that
$K^0_G(\C P(V))$ and $K_G^0(\C P(V))$ are both free modules over $R(G)$ on 
generators which map to each other under restriction. This is explained in 
more detail in \cite{gw:pdiektfc}.

\section{Equivariant Poincar\'{e} duality}

\subsection{Orientation of topological $G$-manifolds}

We work with smooth $G$-manifolds $M$, for which the {\em Slice Theorem}
\cite[II Theorem 5.4]{gb:itctg} asserts that given $x\in M$ with isotropy $G_x\leq G$,
there is a neighbourhood $U$ of the orbit $Gx$, which is $G$-homeomorphic to
$G\times_{G_x}V_x$, where $V_x$ is the tangent space to $M$ at $x$.

\begin{lemma}\label{lemma:G-excision leading to fundamental class}
Using the notation of the Slice Theorem, for each $i$ there are isomorphisms
\begin{enumerate}[(i)]
\item$E^G_i(M,M\setminus Gx)\cong E^G_i(U,U\setminus Gx)$;
\item$E^G_i(U,U\setminus Gx)\cong E^G_i(G\times_{G_x}{V_x},(G\times_{G_x}{V_x})\setminus Gx)$;
\item$E^G_i(G\times_{G_x}{V_x},(G\times_{G_x}{V_x})\setminus Gx)\cong\widetilde{E}^{G}_i(G_+\Smash_{G_x}S^{V_x})$.
\end{enumerate}
\end{lemma}

\begin{proof}
For (i) and (ii), use excision. Part (iii) is equivalent to showing that
$$
E_i^{G_x}({V_x},{V_x}\setminus\{0\})\cong\widetilde{E}^{G_x}_i(S^{V_x}),
$$
and this follows since $S^{V_x}\setminus\{0\}\cong_G V_x$, which
is contractible.
\end{proof}

Composing the three isomorphisms of Lemma \ref{lemma:G-excision
leading to fundamental class}, the outcome is that
\begin{equation}\label{equation:G excision outcome}
E_*^G(M,M\setminus Gx)\cong \widetilde{E}_*^{G_x}(S^{V_x}).
\end{equation}

Provided we restrict to cohomology theories $E_G^*$ and manifolds
$M$ so that the modules $\widetilde{E}_*^{G_x}(S^{V_x})$ that occur 
in this way are free on one generator, we may copy the classical 
definitions.

\begin{definition}[Fundamental classes]\label{definition:G-fundamental class}
\begin{enumerate}[(i)]
\item A cohomology theory $E_G^*(\cdot)$ is said to be {\em complex stable}
if, for each complex representation $V$, there are classes
$\sigma_{V} \in \widetilde{E}_G^{|V|}(S^V)$ giving isomorphisms
$$\widetilde{E}_G^*(S^{|V|}\Smash X) \stackrel{\cong}\lra \widetilde{E}_G^*(S^{V} \Smash X)$$
for any $G$-spectrum $X$. Note in particular that this means
$\widetilde{E}_G^*(S^V)$ is a free $E_G^*$-module on one generator.  
\item Let $M$ be a smooth $G$-manifold of dimension $n$, and let $E_G^*(\cdot) $ be a 
complex stable cohomology theory.  Consider the
composite $\phi_{Gx}$ below. The maps labelled (i), (ii), (iii) are the corresponding isomorphisms of Lemma \ref{lemma:G-excision leading to fundamental class}, $\phi$ is the change of group isomorphism (Lemma \ref{lemma:change of group}) and
$$
i^{Gx}_*:E_*^G(M)\cong E_*^G(M,\emptyset)\lra E_*^G(M,M\setminus Gx)
$$
is the map induced by $G$-inclusion of the $G$-pairs $(M,\emptyset)\stackrel{i^{Gx}}{\lhra}(M,M\setminus Gx)$.
\begin{equation}\label{equation:equivariant fundamental class}
\begin{minipage}{\textwidth}
\begin{xy}
\xymatrix{
E^G_*(M)\ar[d]_-{i^{Gx}_*}\ar[rr]^-{\phi_{Gx}}&&\widetilde{E}^{G_x}_*(S^{V_x})\\
E^G_*(M,M\setminus Gx)\ar[d]^-\cong_-{(i)}&&\widetilde{E}^{G}_*(G_+\Smash_{G_x}S^{V_x})\ar^\cong_{\phi^{-1}}[u]\\
E^G_*(U,U\setminus Gx)\ar[rr]^-\cong_-{(ii)}&&E^G_*(G\times_{G_x}{V_x},(G\times_{G_x}{V_x})\setminus Gx)\ar[u]^-\cong_-{(iii)}
}
\end{xy}
\end{minipage}
\end{equation}
An element $\xi\in E^G_n(M)$ is a {\em fundamental class} for $M$ if the image $\phi_{Gx}(\xi)$ is an $\widetilde{E}^{G_x}_*$-module generator for $\widetilde{E}^{G_x}_*(S^{V_x})$ for all $x\in M$, in which case one writes $[M]$ for such a $\xi$.
\end{enumerate}
\end{definition}

\subsection{Poincar\'{e} duality}

Before we can state the Poincar\'{e} duality theorem we must first recall \cite[III \S3]{ll:esht} how cap products work in the represented setting.

\begin{definition}[Cap products]
Let $E$ be a commutative ring $G$-spectrum with multiplicative structure $\mu$,
and let $X$ be a $G$-$CW$-complex. The {\em cap product}
$E^*_G(X)\otimes E^G_*(X)\lra E_*^G(X)$ is defined by setting $c\cap h$ to be the composite
$$
S\stackrel{h}{\lra}E\Smash X\stackrel{1\Smash\Delta}{\lra}E\Smash X\Smash X\stackrel{1\Smash c\Smash1}{\lra}E\Smash E\Smash X\stackrel{\mu\Smash1}{\lra}E\Smash X.
$$
\end{definition}

\begin{theorem}[Poincar\'{e} duality]
Let $E_G^*(\cdot) $ be a complex stable cohomology theory. If $M$ is a smooth
$G$-manifold with $E_G^*$-fundamental class $[M]$ then there is an isomorphism
$$
E^*_G(M)\stackrel{\cong}{\lra}E_*^G(M)
$$
given by capping with the fundamental class, precisely $a\lmt a\cap[M]$ for $a\in E^*_G(M)$.
\end{theorem}

\begin{proof}
The classical proof (see, for example, \cite[\S26]{mg:atafc}) proceeds by showing that $(-)\cap[M]$ induces an isomorphism on larger and larger subsets of $M$, starting from a point, and using Mayer-Vietoris sequences and excision. The only difference in our case is that we must start with a $G$-point, in other words the orbit $Gx$ for $x\in M$. By definition, the 
fundamental class provides exactly this input.
\end{proof}

\section{Construction of the fundamental class}

\subsection{Equivariant $K$-theory of $\C P(V)$}
\label{subsec:KCPV}

Our computation of $K^*_G(\C P(V))$ arises from the based cofibre sequence
\begin{equation}\label{equation:cofibre sequence}
S(V\otimes z)_+\lra D(V\otimes z)_+\lra D(V\otimes z)/S(V\otimes z)\cong S^{V\otimes z}
\end{equation}
and the following fundamental result of Atiyah and Segal \cite{gs:ekt}.

\begin{theorem}\label{theorem:K of a free action}
Let $G$ be a compact Lie group. Suppose $N$ is a normal subgroup which acts freely on the $G$-$CW$-complex $X$. Then the quotient $X\lra X/N$ induces an isomorphism $K^*_{G/N}(X/N)\stackrel{\cong}{\lra}K_G^*(X)$.
\qed
\end{theorem}

Applying $\K^*_{G\times T}(-)$ to (\ref{equation:cofibre sequence}) and appealing to Theorem \ref{theorem:K of a free action} gives the long exact sequence
$$
\cdots\lra\K^0_{G\times T}(S^{V\otimes z})\lra\K^0_{G\times T}\lra\K^0_G(\C P(V)_+)\lra\K^1_{G\times T}(S^{V\otimes z})\lra\cdots.
$$

\begin{proposition}
We have $K^0_G(\C P(V))\cong\frac{R(G)[z]}{\chi(V\otimes z)}$.
\end{proposition}

\begin{proof}
We claim that the long exact sequence above gives a short exact sequence
\begin{equation}\label{equation:ses}
0\lra R(G\times T)\stackrel{\psi}{\lra}R(G\times T)\lra\K^0_G(\C P(V)_+)\lra0.
\end{equation}
Indeed, by equivariant Bott periodicity
$\K^1_{G\times T}(S^{V\otimes z})\cong\K^1_{G\times T}(S^0)=0$.
The Thom isomorphism tells us
that $\K^0_{G\times T}(S^{V\otimes z})\cong K^0_{G\times T}$
and, by definition of the Euler class, $\im(\psi)$ is the ideal generated by $\chi(V\otimes z)$. The fact that multiplication by the Euler class is injective in (\ref{equation:ses})
follows since $\widetilde{K}^{-1}_{G \times T}(S^{V\tensor z})=0$. 
The first isomorphism theorem now tells us that
$$
\K^0_G(\C P(V)_+)\cong\frac{R(G\times T)}{\chi(V\otimes z)},
$$
and we observe \cite{ja:lolg} that $R(G\times T)\cong R(G)[z,z^{-1}]$, from which the proposition follows.
\end{proof}

When we come to consider homology, the Adams isomorphism takes the role of Theorem \ref{theorem:K of a free action} and we have a subtle dimension shift, viz
$$
\K_0^G(\C P(V)_+)\cong\K_{-1}^{G\times T}(S(V\otimes z)_+).
$$
Excepting this technical point, we find in a similar fashion a short exact sequence
\begin{equation}\label{equation:homology ses}
0\lra R(G\times T)\stackrel{\psi}{\lra}R(G\times T)\lra\K_0^G(\C P(V)_+)\lra0,
\end{equation}
in which $\psi$ is again multiplication by the Euler class.

We now choose a notation which will be convenient for comparing results for projective
spaces of different representations in \S\ref{section:calculations}.

\begin{proposition}
We have $K_0^G(\C P(V))\cong\frac{\frac{1}{\chi(V\otimes z)}R(G\times T)}{R(G\times T)}$,
where $\frac{1}{\chi(V\otimes z)}R(G\times T)$ is the $R(G\times T)$-submodule
generated by $\frac{1}{\chi (V\otimes z)}$ in the total ring of fractions of
$R(G\times T)$.
\end{proposition}

\begin{proof}
Just replace the short exact sequence (\ref{equation:homology ses}) with
the isomorphic short exact sequence
\begin{equation}\label{equation:homology better ses}
0\lra R(G\times T)\lhra\frac{1}{\chi(V\otimes z)}R(G\times T)\lra\K_0^G(\C P(V)_+)\lra0.
\end{equation}
\end{proof}


\subsection{Duality from the Universal Coefficient Theorem.}

It is convenient to record a simple case of the algebraic relation between
homology and cohomology. For any ring $G$-spectrum $E$ and any $G$-spectrum
$Y$ we have a natural map
$$
p_Y: E^*_G(Y){\lra}\hom_{E_*^G}{(E^G_*(Y),E_*^G)}.
$$
A suitable Universal Coefficient Theorem (UCT) would state that
$p_Y$ is an isomorphism if $E^G_*(X)$ is projective as an
$E_*^G$-module. In equivariant topology the existence of such a
UCT is more than the formality it is non-equivariantly
\cite{ae:rmaaisht}, for a variety of linked reasons. From one
point of view, the issue is that on the one hand the usual
building blocks of $G$-spaces are the orbits $G/H$, whilst on the
other $E_G^*(G/H)\cong E_H^*$ is unlikely to be projective. For
these reasons, the sort of UCT that exists for formal reasons
\cite{LM, mj:hcfekt} is based on Mackey functor valued homology
and cohomology. Since this does not directly discuss $p_Y$,
additional work is required, which relies upon special properties
of the cohomology theory, or the group of equivariance, or the
space. For $K$-theory, one does expect a UCT for general
$G$-spaces, but for present purposes we will be content to prove
the very special case that concerns us.

\begin{lemma}\label{lemma:homology and cohomology are dual}
If $X=\C P (V)$ then
\label{lemma:homology and cohomology are dual:enumi:1}
we have isomorphisms
$$
K^*_G(X)\stackrel{\cong}{\lra}\hom_{K_*^G}{(K^G_*(X),K_*^G)}
$$
and
$$
K_*^G(X)\stackrel{\cong}{\lra}\hom_{K^*_G}{(K^*_G(X),K^*_G)}.
$$
\end{lemma}

\begin{proof}

Taking $E=K$, $p_X$ gives the first comparison map, and applying Spanier-Whitehead duality to $p_{DX}$ gives the second.

First, we prove that if $V$ is a sum of one dimensional representations
the map $p_X$ is an isomorphism. The same argument shows $p_{DX}$
is an isomorphism. We argue by induction on the dimension of $V$. If
$V$ is one dimensional then $\C P (V)$ is a point and the conclusion
is clear. Now suppose that  $V=W\oplus \alpha$ with
$\alpha $ one dimensional, and that $p_{\C P(W)}$ is known to be
an isomorphism.  There is a cofibre sequence
$$\C P (W) \lra \C P (V) \lra S^{W \tensor \alpha^{-1}}, $$
which induces a short exact sequence of free $K_G^*$-modules
in both homology and cohomology. Since $p_{S^{W \tensor \alpha^{-1}}}$
is an isomorphism, we conclude $p_{\C P (V)}$ is an isomorphism as
required.

This shows that $p_{\C P (V)}$ is an isomorphism for all $V$ if
$G$ is abelian, and we now consider the general case. We have a commutative square
$$
\begin{xy}
\xymatrix{
K^*_G(X)\ar[r]^-{p_X}\ar[d]&\hom_{K_*^G}{(K^G_*(X),K_*^G)}\ar[d]\\
\prod\limits_{\genfrac{}{}{0pt}{}{H\leq G}{H\mathrm{\ cyclic}}}K^*_H(X)\ar[r]^-\cong&\prod\limits_{\genfrac{}{}{0pt}{}{H\leq G}{H\mathrm{\ cyclic}}}\hom_{K_*^H}{(K^H_*(X),K_*^H)}.
}
\end{xy}
$$

\noindent Since the left hand
vertical is the monomorphism (\ref{equation:injective restriction}), it follows that $p_X$ is a monomorphism.
The same applies to $p_{DX}$.

We also have a commutative square
$$
\begin{xy}
\xymatrix{
K_*^G(X)\ar[r]^-\cong\ar[d]_\cong&K_*^G(D^2X)\ar[d]\\
\hom_{K_*^G}(\hom_{K_*^G}(K_*^G(X),K_*^G),K_*^G)\ar[r]^-{(p_X)^*}&\hom_{K_*^G}(K^*_G(X),K_*^G).
}
\end{xy}
$$

\noindent The top horizontal is an isomorphism because $X$ is
finite, so that the natural map $X \stackrel{\simeq}\lra D^2X$ is
an equivalence. The left hand vertical is an isomorphism because
$K_*^G(X)$ is a finitely generated free module. The right hand
vertical is $p_{DX}$, combined with Spanier-Whitehead duality, so
that the composite obtained by travelling the square first
horizontally, then vertically, is the second comparison map. This
shows that the second comparison map is the algebraic dual of
$p_X$. Since $p_X$ is a monomorphism, duality shows that the
second comparison map is an epimorphism, and hence an
isomorphism. The first comparison map is dealt with similarly.
\end{proof}

\begin{remark}
There is an alternative approach to the duality statement which is perhaps more illuminating
from the algebraic point of view. Writing $R=R(G)$ and $S=R(G\times T)$, and $\chi =
\chi (V\otimes z)$
we calculated the homology
$$K^{G}_0(\Sigma^2\C P (V)) =S/\chi$$
from the short exact sequence arising from the sequence of $G\times T$-spaces
$S^0 \lra S^{V\otimes z} \lra \Sigma S(V\otimes z)_+$,
which we regard as a projective resolution over $S$. This means that the cohomology
of $S(V\otimes z)_+ \lra S^0 \lra S^{V\otimes z}$ shows
$$K_{G}^0(\C P (V)) =\ext^1_S(S/\chi,S).$$
Thus the UCT duality statement is
$$\ext^1_S(S/\chi,S) \cong \hom_R(S/\chi ,R), $$
 and one can write down the isomorphism explicitly in these terms. Furthermore,
the short exact sequence
$$0 \lra S \stackrel{\chi}\lra S \lra \ext^1_S(S/\chi,S) \lra 0$$
can be viewed as an exact sequence of $R$-modules; since the $R$-modules
are all free, applying $(\cdot)^*=\hom_R(\cdot , R)$
we see the more elementary isomorphism
$$\ext^1_S(S/\chi,S)^* \cong \hom_R(S/\chi , R)$$
which corresponds to Poincar\'e duality.
By contrast with topology, from the algebraic point of view, it is the UCT that
is the more subtle statement, and Poincar\'e duality that is formal.
\end{remark}

\subsection{The fundamental class}
The following identification of the fundamental class is the key result of the paper.
\begin{theorem}\label{theorem:fundamental class}
Let $G$ be a finite group and $V$ a complex representation of $G$ with $\dim_\C(V)=n$. Then $\frac{1}{\chi(V\otimes z)}\in K_0^G(\C P(V))$ is a fundamental class in equivariant $K$-theory for $\C P(V)$.
\end{theorem}

We break our proof into convenient pieces as follows. For brevity we write $\notensor{V}{z}$ for $V\otimes z$, {\em etc}.

\begin{lemma}\label{lemma:res(1/chi)=1/chi}
The notation is compatible with restriction, in the sense that for any
subgroup $H$ of $G$, we have
 $\res{G}{H}\left(\frac{1}{\chi(\notensor{V}{z})}\right)=\frac{1}{\chi(\notensor{V}{z})}$.
\end{lemma}

\begin{proof}
If we use \ref{equation:homology ses} to say $K^G_0(\C P (V))=R(G\times T )/(\chi (Vz))$
the element $1/\chi (Vz)$ corresponds to the unit of $R(G\times T)$. The
lemma  simply states that the restriction of the unit in $R(G\times T)$ is the unit in 
$R(H \times T)$.
\end{proof}

\begin{lemma}\label{lemma:step 1}
If $x\in\C P(V)$ is $G$-fixed, then $i^{Gx}_*(\frac{1}{\chi(\notensor{V}{z})})$
is an $R(G\times T)$-generator for $K_0^G(\C P(V),\C P(V)\setminus Gx)$.
\end{lemma}

\begin{proof}
The point $x$ represents a line in $V$, and since it is fixed, this is
 a 1-dimensional representation $\alpha$ of $G$, and we have
$V\cong W\oplus\alpha$ for some $W$. Thus
$$
\C P(V)\setminus Gx=\C P(V)\setminus\C P(\alpha)\simeq \C P(W),
$$
so we are required to prove that $K_{-1}^{G\times T}(S(\notensor{V}{z}),S(\notensor{W}{z})))$ is $R(G\times T)$-generated by $i^{Gx}_*(\frac{1}{\chi(\notensor{V}{z})})$.

We have a commutative diagram
$$
\begin{xy}
\xymatrix{
&&0\ar[d]\\
0\ar[r]&\K_0^{G\times T}(S^0)\ar[r]^-{\chi(\notensor{W}{z})}\ar@{=}[d]&\K_0^{G\times T}(S^{\notensor{W}{z}})\ar[r]\ar[d]^{\chi(\notensor{\alpha}{z})}&\K_{-1}^{G\times T}(S(\notensor{W}{z})_+)\ar[r]\ar[d]&0\\
0\ar[r]&\K_0^{G\times T}(S^0)\ar[r]^-{\chi(\notensor{V}{z})}&\K_0^{G\times T}(S^{\notensor{V}{z}})\ar[r]^-{a}\ar[d]_{b}&\K_{-1}^{G\times T}(S(\notensor{V}{z})_+)\ar[r]\ar[d]^{c}&0\\
&&K_{-1}^{G\times T}(S(\notensor{V}{z}),S(\notensor{W}{z}))\ar[d]\ar@{=}[r]&K_{-1}^{G\times T}(S(\notensor{V}{z}),S(\notensor{W}{z}))\\
&&0
}
\end{xy}
$$
in which the rows and columns are exact. (The rows are (\ref{equation:homology ses}), the centre column is the homology sequence of the $G$-triple $(D(\notensor{V}{z}),S(\notensor{V}{z}),S(\notensor{W}{z}))$ and the right-hand column comes from the $G$-pair $(S(\notensor{V}{z}),S(\notensor{W}{z}))$.)
Writing $\beta$ for the Bott class in $\K_0^{G\times T}(S^{\notensor{V}{z}})$, we must show that
$ca(\beta)=b(\beta)$ is an $R(G\times T)$-generator. This is clear, since $\beta$
obviously $R(G\times T)$-generates $\K_0^{G\times T}(S^{\notensor{V}{z}})$.
\end{proof}

\begin{lemma}\label{lemma:step 2}
Under the hypothesis of Lemma \ref{lemma:step 1}, $i^{Gx}_*(\frac{1}{\chi(\notensor{V}{z})})$ is an $R(G)$-generator for $K_0^G(\C P(V),\C P(V)\setminus Gx)$.
\end{lemma}

\begin{proof}
We work in cohomology, where the module structure is transparent, and the result in homology
 follows via duality. Our proof now amounts to showing that the action of
$z\in R(G\times T)$ on $K^0_{G\times T}(S(\notensor{V}{z}),S(\notensor{W}{z}))$ is the same as that
of $\alpha^{-1}\in R(G)$. We have an equivalence
$$
S(\notensor{V}{z})/S(\notensor{W}{z})\simeq S(\notensor{\alpha}{z})_+\Smash S^{\notensor{W}{z}},
$$
and writing $\kappa=\ker{(\notensor{\alpha}{z})}$ we have $S(\notensor{\alpha}{z})\cong(G\times T)/\kappa$
so we may work in $K^0_\kappa(S^{\notensor{W}{z}})$. Finally, we identify $\kappa$ with $G$ by
the isomorphism $f:G\stackrel{\cong}{\lra}\kappa$ defined by $f(g)=(g,\alpha (g)^{-1})$.
Thus $f^*(\notensor{W}{z})=\notensor{W}{\alpha^{-1}}$. The result now follows by considering the
 commutative diagram
$$
\begin{xy}
\xymatrix{
\K^0_{G\times T}\times \K^0_{G\times T}((G\times T)/\kappa_+\Smash S^{\notensor{W}{z}})\ar[rr]^-{m}\ar[d]_{\res{G\times T}{\kappa}\times\theta^{-1}}&&\K^0_{G\times T}((G\times T)/\kappa_+\Smash S^{\notensor{W}{z}})\ar[d]_\cong^{\theta^{-1}}\\
\K^0_\kappa\times \K^0_\kappa(S^{\notensor{W}{z}})\ar[rr]^-{m}\ar[d]_{f^*}^\cong&&\K^0_\kappa(S^{\notensor{W}{z}})\ar[d]^{f^*}_\cong\\
\K^0_G\times \K^0_G(S^{\notensor{W}{\alpha^{-1}}})\ar[rr]^-{m}&&\K^0_G(S^{\notensor{W}{\alpha^{-1}}}),
}
\end{xy}
$$
in which $m$ is the module structure.
\end{proof}

\begin{lemma}\label{lemma:step 3}
Suppose $x\in \C P(V)$ has isotropy $H=G_x<G$. Given $B\subseteq A\subseteq \C P(V)$, write $i^A_B$ for the inclusion of $G$-pairs
$$
i^A_B:(\C P(V),\C P(V)\setminus A)\lhra(\C P(V),\C P(V)\setminus B).
$$
Writing $i^B$ for $i^{\C P(V)}_B$, we have a commutative diagram
\begin{equation}\label{equation:final diagram}
\begin{minipage}{\textwidth}
\begin{xy}
\xymatrix{
K_0^G(\C P(V))\ar[rr]^-{i^{Gx}_*}\ar[d]_{\res{G}{H}}&&K_0^G(\C P(V),\C P(V)\setminus Gx)\ar[d]^{\res{G}{H}}\ar[rr]^-{\cong}_-{\mbox{{\scriptsize (\ref{equation:G excision outcome})}}}&&K_0^{H}\ar@{=}[dd]\\
K_0^{H}(\C P(V))\ar@{=}[d]\ar[rr]^-{i^{Gx}_*}&&K_0^{H}(\C P(V),\C P(V)\setminus Gx)\ar[d]^{(i^{Gx}_{Hx})_*}\\
K_0^{H}(\C P(V))\ar[rr]^-{i^{Hx}_*}&&K_0^{H}(\C P(V),\C P(V)\setminus Hx)\ar[rr]^-{\cong}_-{\mbox{{\scriptsize (\ref{equation:G excision outcome})}}}&&K_0^{H}.
}
\end{xy}
\end{minipage}
\end{equation}
\end{lemma}

\begin{proof}
Commutativity of the left hand squares in the diagram is obvious by naturality.
For the right hand square,  use Lemma \ref{lemma:G-excision leading to fundamental class}
to write out the isomorphism (\ref{equation:G excision outcome}) in full.
\end{proof}

\begin{proof}[Proof of Theorem \ref{theorem:fundamental class}.]
Equivariant Bott periodicity means that we can work everywhere in degree zero. Let $x\in\C P(V)$. Suppose $x$ has isotropy $H\leq G$. By Lemma \ref{lemma:step 3} it suffices to show that $(i^{Hx})_*\res{G}{H}(\frac{1}{\chi(\notensor{V}{z})})$ is a generator. Now Lemma \ref{lemma:res(1/chi)=1/chi} allows us to use Lemma \ref{lemma:step 2} to complete the proof.
\end{proof}

\section{Calculations with the fundamental class}\label{section:calculations}

\subsection{The abelian world}\label{section:abelian}

For the time being, let us impose the restriction that $G$ be a finite {\em abelian} group $A$. Given an $n$-dimensional complex representation $V$ of $A$, we can write $V=\alpha_1\oplus\cdots\oplus\alpha_n$ for one dimensional summands $\alpha_i$. Following \cite{mc:tuoecb} we {\em choose} a complete flag
\begin{equation}\label{equation:flag}
\cur{F}=\left(0\subset V^1\subset V^2\subset\cdots\subset V^n=V\right)
\end{equation}
in which $V^i/V^{i-1}=\alpha_i$. This choice gives rise to an $R(A)$-basis
$$
\{1,y^{V^1},y^{V^2},\ldots,y^{V^{n-1}}\}
$$
for $K^0_A(\C P(V))$, in which
$$y^{V^i}=y^{\alpha_1}y^{\alpha_2} \cdots y^{\alpha_i} \mbox{ and }
y^{\alpha_j}=1-\alpha_jz.$$
We write $\{\beta_0^\cur{F},\ldots,\beta_{n-1}^\cur{F}\}$ for the dual $R(A)$-basis
for
$$K_0^A(\C P(V))\cong\hom_{R(A)}(K^0_A(\C P(V)),R(A)), $$
so that
$$
\beta_i^{\cur{F}}(y^{V^j})=\delta_i^j.
$$

\begin{theorem}\label{theorem:sum of betas is 1/chi}
The fundamental class is given by
\begin{equation}\label{equation:sum of betas is 1/chi}
\frac{1}{\chi(V\otimes z)}=\beta^\cur{F}_0+\cdots +\beta^\cur{F}_{n-1}.
\end{equation}
\end{theorem}

\begin{remarks} \label{remarks:summary}
\begin{enumerate}[(i)]
\item Since the left hand side of (\ref{equation:sum of betas is 1/chi})
is a topological invariant
of $V$, so too is the right hand side and we may abbreviate  to
$\beta_0+\cdots+\beta_{n-1}$ without ambiguity. It is striking that
although the individual $\beta_i^\cur{F}$ depend on the flag $\cur{F}$, this
sum does not.
\item This generalises Adams's classical identification of the (non-equivariant)
$K$-theory fundamental class \cite[Theorem III.11.15]{ja:shagh}, and provides
a more elementary
proof (Adams's alternating signs arise by choosing the opposite orientation).
\item\label{remarks:summary:enumi:3} One can give a direct, algebraic proof that
$\beta_0+\cdots+\beta_{n-1}$ is independent of flag, without relating it to
$\frac{1}{\chi(V\otimes z)}$. We refer to \cite[Proposition 3.5.13]{gw:pdiektfc} for
details.  Furthermore, in  \cite{gw:pdiektfc} it is shown directly that taking the cap
product with $\beta_0+\cdots+\beta_{n-1}$ gives a duality isomorphism.
\end{enumerate}
\end{remarks}

\begin{proof}[Proof of Theorem \ref{theorem:sum of betas is 1/chi}]
It suffices to prove the result if $A$ is the $n$-torus $T^n$ and $V=z_1\oplus\cdots\oplus z_n$, where $z_i(\lambda_1,\ldots,\lambda_n)=\lambda_i$. This is because the pullback of $z_1\oplus\cdots\oplus z_n$ along the homomorphism $\boldalpha:A\lra T^n$, in which $\boldalpha(a)=(\alpha_1(a),\ldots,\alpha_n(a))$, is $\alpha_1\oplus\cdots\oplus\alpha_n$.

The proof proceeds by induction on $n=\dim_\C(V)$. The initial step is obvious, so now suppose the theorem holds for representations of dimension smaller than $n>1$. For $1\leq i\leq n$, we have $T^n$-inclusions
$$
j_i:\C P(z_1\oplus\cdots\oplus z_{i-1}\oplus z_{i+1}\oplus\cdots\oplus z_n)\lhra\C P(V),
$$
and we write
\begin{eqnarray*}
\iota_i&=&\frac{1}{\chi((z_1\oplus\cdots\oplus z_{i-1}\oplus z_{i+1}\oplus\cdots\oplus z_n)\otimes z)}\\
&\in&K_0^{T^n}(\C P(z_1\oplus\cdots\oplus z_{i-1}\oplus z_{i+1}\oplus\cdots\oplus z_n)).
\end{eqnarray*}

Writing $\langle-,-\rangle$ for the Kronecker pairing we have
\begin{eqnarray}
\nonumber\langle y^{z_1}\cdots y^{z_i},(j_n)_*(\iota_n)\rangle&=&\langle(j_n)^*(y^{z_1}\cdots y^{z_i}),\iota_n\rangle\\
\nonumber&=&\langle y^{z_1}\cdots y^{z_i},\iota_n\rangle\\
\label{equation:using inductive hypothesis}&=&\left\{\begin{array}{ll}1&0\leq i\leq n-2\\0&i=n-1\end{array}\right.,
\end{eqnarray}
from which
\begin{equation}\label{equation:first inductive}
(j_n)_*(\iota_n)=\beta_0^\cur{F}+\cdots+\beta_{n-2}^\cur{F}.
\end{equation}
In the final step of (\ref{equation:using inductive hypothesis}), we use the inductive
hypothesis for $0\leq i \leq n-2$, and for $i=n-1$, the fact that 
$y^{z_1}y^{z_2}\cdots y^{z_{n-1}}=0$. Similarly, one finds that
\begin{equation}\label{equation:second inductive}
(j_{n-1})_*(\iota_{n-1})=\beta_0^\cur{F}+\cdots+\beta_{n-2}^\cur{F}+(1-z_{n-1}z_n^{-1})\beta_{n-1}^\cur{F}.
\end{equation}
Taking a linear combination of 
 (\ref{equation:first inductive}) and (\ref{equation:second inductive}), we find
$$
(j_{n-1})_*(\iota_{n-1})-z_{n-1}z_n^{-1}(j_{n})_*(\iota_{n})=
(1-z_{n-1}z_n^{-1})(\beta_0^\cur{F}+\cdots+\beta_{n-1}^\cur{F})
$$
We now simplify the left hand side, using the fact that
$$
(j_n)_*(\iota_n)=\frac{1}{\chi ((V/z_n)\otimes z)}=
\frac{\chi(z_n\otimes z)}{\chi(V\otimes z)}
$$
and similarly for $(j_{n-1})_*(\iota_{n-1})$.
Since
$$
\frac{\chi(z_{n-1}\otimes z)}{\chi(V\otimes z)}-
z_{n-1}z_n^{-1}\frac{\chi(z_n\otimes z)}{\chi(V\otimes z)}=
(1-z_{n-1}z_n^{-1})\frac{1}{\chi(V\otimes z)}, 
$$
we obtain
$$
(1-z_{n-1}z_n^{-1})\frac{1}{\chi(V\otimes z)}=
(1-z_{n-1}z_n^{-1})(\beta_0^\cur{F}+\cdots+\beta_{n-1}^\cur{F}).
$$
The result follows, since $1-z_{n-1}z_n^{-1}$ is not a zero divisor in $R(T^n\times T)$.
\end{proof}

\subsection{The non-abelian world}\label{section:non-abelian}

The proofs of \S\ref{section:abelian} break down in the non-abelian case because $V$ may not have a decomposition into one-dimensional representations and we cannot choose a flag as in (\ref{equation:flag}).

\begin{notation}
Recall that $K^0_G(\C P(V))\cong R(G)[z]/(\chi(V\otimes z))$ (irrespective of whether $G$ is abelian). Observe that
$$
\cur{B}=\{(1-z)^i~|~0\leq i\leq n-1\}
$$
is always a basis for $K^0_G(\C P(V))$. (Whereas the construction of \S
 \ref{section:abelian}
gives a basis for any complex orientable theory, the fact that $\cur{B}$ gives a basis
is a special feature of $K$-theory).
We write $\{\beta_0^\cur{B},\ldots,\beta_{n-1}^\cur{B}\}$ for the corresponding dual basis for $K_0^G(\C P(V))$.
\end{notation}

Happily, it turns out that in the abelian case, the  explicit proof mentioned in
Remarks \ref{remarks:summary} (\ref{remarks:summary:enumi:3})
gives
\begin{equation}\label{equation:betas and cur(B)}
\frac{1}{\chi(V\otimes z)}=\sum\limits_{i=0}^{n-1}\beta_i=\sum\limits_{i=0}^{n-1}\beta_i^\cur{B}.
\end{equation}

\begin{lemma}\label{lemma:res(beta)=beta}
If $H\leq G$ then we have $\res{G}{H}(\beta_i^\cur{B})=\beta_i^\cur{B}$ for $i=0,\ldots,n-1$.\qed
\end{lemma}

\noindent The proof involves considering the interaction of restriction with the Kronecker pairing. Details may be found in \cite[\S4.4]{gw:pdiektfc}.

\begin{theorem}
Let $V$ be a complex representation, $\dim_\C{V}=n$, of the finite group $G$. Take $\cur{B}=\{(1-z)^i\}_{i=0}^{n-1}$ as a basis for $K^0_G(\C P(V))$ and let the dual basis for $K_0^G(\C P(V))$ be $\{\beta_i^\cur{B}\}_{i=0}^{n-1}$. Then
$$
\sum\limits_{i=0}^{n-1}\beta_i^\cur{B}=\frac{1}{\chi(V\otimes z)}.
$$
\end{theorem}

\begin{proof}
We use Lemma \ref{lemma:res(beta)=beta} to see that $\res{G}{H}(\sum\limits_{i=0}^{n-1}\beta_i^\cur{B})=\sum\limits_{i=0}^{n-1}\beta_i^\cur{B}$ and Lemma \ref{lemma:res(1/chi)=1/chi} to see that $\res{G}{H}(\frac{1}{\chi(V\otimes z)})=\frac{1}{\chi(V\otimes z)}$ for each $H\leq G$. Taking the product over cyclic subgroups, and using (\ref{equation:betas and cur(B)}),
$$
\res{G}{*}(\sum\limits_{i=0}^{n-1}\beta_i^\cur{B})=\res{G}{*}(\frac{1}{\chi(V\otimes z)}).
$$
The theorem now follows from the injectivity of $\res{G}{*}$.
\end{proof}

\subsection{Perfect pairings}

Recall that if $M,N$ are modules over the commutative ring $R$ then a bilinear map $b:M\otimes N\lra R$ is a {\em perfect pairing} if
$$\begin{array}{ccc}
M & \lra &\hom_R(N,R)\\
m & \lmt &(n\lmt b(m\otimes n))
\end{array}
$$
defines an isomorphism of $R$-modules $M\stackrel{\cong}{\lra}\hom_R(N,R)$.

\begin{notation}\label{notation:pairing}
We define a pairing $\lceil-,-\rceil:K^0_G(\C P(V))\otimes K^0_G(\C P(V))\lra R(G)$ by $\lceil x,y\rceil=\langle xy,\frac{1}{\chi(V\otimes z)}\rangle$.
\end{notation}

\begin{theorem}
The pairing
$$
\lceil-,-\rceil:K^0_G(\C P(V))\otimes K^0_G(\C P(V))\lra R(G)
$$
is perfect, and the corresponding isomorphism
$$
K^0_G(\C P(V))\stackrel{\cong}{\lra}\hom_{R(G)}(K^0_G(\C P(V)),R(G))=K_0^G(\C P(V))
$$
is a Poincar\'{e} duality isomorphism.
\end{theorem}

\begin{proof}
One can show directly that $\lceil-,-\rceil$ is perfect in the abelian case, but it is far more satisfactory (and general) to observe that the map
$$
K^0_G(\C P(V))\stackrel{\cap_\xi}{\lra}\hom_{R(G)}(K^0_G(\C P(V)),R(G))=K_0^G(\C P(V)),
$$
in which $x\stackrel{\cap_\xi}{\lmt}(y\lmt\langle xy,\xi\rangle)$, is
capping with $\xi\in K_0^G(\C P(V))$ -- in other words $\cap_\xi(x)=x\cap\xi$. This is
easily verified, using Lemma \ref{lemma:homology and cohomology are dual} and the definition of the cap product.
\end{proof}

\section{Examples}

We conclude by explaining  how to compute the pairing $\lceil-,-\rceil$ of
Notation \ref{notation:pairing} for any $\C P(V)$. We make the results explicit
in dimensions $\leq 4$.

As observed above, $K^0_G(\C P(V))\cong R(G)[z]/\chi (Vz)$, and
we use the basis  $\{1,y,y^2, \ldots , y^{n-1}\}$ if $V$ is of dimension $n$,
where $y=1-z$.  As described above $\lceil a,b \rceil =\eps (ab)$
where
$$\eps (a_0 +a_1y +\cdots + a_{n-1}y^{n-1})=a_0+ a_1 + \cdots +a_{n-1}\in R(G).$$
Given $s\geq0$, we therefore need to find expressions for $y^{n+s}$ in terms of the basis:
in fact if
$$y^{n+s}=\sum_{j=0}^{n-1}\lambda^s_jy^j,$$
we will find recursive formulae for $\lambda^s_j$, and then
$$\lceil y^i,y^j \rceil =\eps (y^{i+j})=\lambda_0^s+ \cdots + \lambda_{n-1}^s,$$
if $i+j=n+s$.

We first apply the splitting principle to obtain a formula for $\chi (Vz)$ in
terms of $y$, and we use notation suggested by the theory of equivariant
formal group laws.  Indeed if $\alpha$ is one dimensional,
$$\chi (\alpha z)=1-\alpha z=e(\alpha)+\alpha y =\alpha (y-e(\alpha^{-1})),$$
where $e(\alpha) =1-\alpha$. Now, if $V=\alpha_1\oplus\cdots\oplus\alpha_{n}$ is a 
sum of one dimensional representations,
$$\det (V)^{-1} \chi (Vz)=\prod_{i=1}^n (y-e(\alpha_i^{-1}))=\sigma_n+\sigma_{n-1}y+
\cdots + \sigma_1y^{n-1}+y^n,$$
where we have used the elementary symmetric polynomials
$$\sigma_j=\sigma_j(-e(\alpha_1^{-1}), -e(\alpha_2^{-1}), \ldots, -e(\alpha_n^{-1})).$$
Since the $\sigma_j$ are symmetric, the coefficients can be expressed
in terms of exterior powers. Explicitly, writing $V^*$ for the dual representation of $V$, we have the formula
\begin{multline*}
\sigma_m=\lambda^{m}(V^*)-\bin{n-m+1}{n-m}\lambda^{m-1}(V^*)
+\bin{n-m+2}{n-m} \lambda^{m-2}(V^*) -\cdots\\
\cdots +(-1)^{m-1}\bin{n-1}{n-m}\lambda^{1}(V^*)+(-1)^m \bin{n}{n-m}.
\end{multline*}
Thus we have an equality
$$\det (V)^{-1} \chi (Vz)=\sigma_n+\sigma_{n-1}y+
\cdots + \sigma_1y^{n-1}+y^n, $$
between elements of $R(G\times T)$: we have verified it when $V$ is
a sum of one dimensional representations, and it therefore holds
in general by the splitting principle.

Thus the condition $\chi (Vz)=0$ is equivalent to
$$y^n=-(\sigma_n+\sigma_{n-1}y+ \cdots + \sigma_1y^{n-1}),$$
or $\lambda_j^0=-\sigma_{n-j}$.
Now
$$y^{n+s+1}=yy^{n+s}=\sum_{j=1}^{n-1}\lambda^s_{j-1}y^j -
\lambda_{n-1}^s\sum_{j=0}^{n-1}\sigma_{n-j}y^j, $$
or, interpreting $\lambda^s_{-1}$ as zero,
$$\lambda_j^{s+1}=\lambda_{j-1}^s-\lambda_{n-1}^s\sigma_{n-j}.$$
When adding up, it is useful to note that $1-e(\alpha)=\alpha$, so in particular
$$\det (V)^{-1}=1+\sigma_1+\sigma_2 + \cdots + \sigma_n.$$
Then we find
$$\eps (y^n)=1-\det(V)^{-1}.$$
Similarly,
$$\eps (y^{n+s+1})=\eps(y^{n+s})-\lambda_{n-1}^s\det(V)^{-1},$$
and an inductive argument then shows
$$\eps(y^{n+s})=1-\frac{1}{\det (V)}(1+\lambda_{n-1}^0+ \cdots +\lambda_{n-1}^{s-1}).$$
More explicitly, if we interpret $\sigma_{n+s}$ as zero for $s>0$,
$$\lambda_j^0=-\sigma_{n-j},\; \lambda_j^1=-\sigma_{n-j+1}+\sigma_1 \sigma_{n-j},\;
\lambda_j^2=-\sigma_{n-j+2}+\sigma_1 \sigma_{n-j+1}
+(\sigma_2-\sigma_1^2)\sigma_{n-j},$$
and so
\begin{align*}
\eps (y^n)&=1-\frac{1}{\det (V)},\\
\eps (y^{n+1})&=1-\frac{1}{\det (V)}(1-\sigma_1),\\
\eps (y^{n+2})&=1-\frac{1}{\det (V)}(1-(\sigma_1+\sigma_2)+\sigma_1^2),\\
\eps (y^{n+3})&=1-\frac{1}{\det (V)}(1-(\sigma_1+\sigma_2+\sigma_3)+
(2\sigma_1\sigma_2+\sigma_1^2)-\sigma_1^3).
\end{align*}

Below are the results of the pairing $\lceil-,-\rceil$ for $\C P (V)$
when $V$  is of small dimension. (For brevity, we write $\delta^*$ for
$\det (V^*)=1/\det (V)$).

\begin{center}
\begin{tabular}{|c|cc|}\hline
&$1$&$y$\\\hline
$1$&$1$&$1$\\
$y$&$1$&$1-\delta^*$\\\hline
\multicolumn{3}{c}{}\\
\multicolumn{3}{c}{Pairing for  $\dim (V)=2$}
\end{tabular}
\end{center}

\begin{center}
\begin{tabular}{|c|ccccc|}\hline
&$1$&&$y$&&$y^2$\\\hline
$1$&$1$&&$1$&&$1$\\
$y$&$1$&&$1$&&$1-\delta^*$\\
$y^2$&$1$&&$1-\delta^*$&&$1-\delta^*(4-V^*)$\\\hline
\multicolumn{6}{c}{}\\
\multicolumn{6}{c}{Pairing  for $\dim V=3$}
\end{tabular}
\end{center}

\begin{center}
\begin{tabular}{|c|ccccccc|}\hline
&$1$&&$y$&&$y^2$&&$y^3$\\\hline
$1$&$1$&&$1$&&$1$&&1\\
$y$&$1$&&$1$&&$1$&&$1-\delta^*$\\
$y^2$&$1$&&$1$&&$1-\delta^*$&&$1-\delta^*(5-V^*)$\\
$y^3$&$1$&&$1-\delta^*$&&$1-\delta^*(5-V^*)$&&$1-\delta^*(14-6V^*+(V^*)^2-\lambda^2(V^*))$\\
\hline
\multicolumn{8}{c}{}\\
\multicolumn{8}{c}{Pairing  for $\dim V=4$}
\end{tabular}
\end{center}


\begin{thebibliography}{10}

\bibitem{ja:poeshfcl}
J.~F. Adams.
\newblock Prerequisites (on equivariant stable homotopy) for {C}arlsson's
  lecture.
\newblock In {\em Algebraic topology, Aarhus 1982 (Aarhus, 1982)}, volume 1051
  of {\em Lecture Notes in Mathematics}, pages 483--532. Springer, Berlin, 1984.

\bibitem{ja:lolg}
J.~F. Adams.
\newblock {\em Lectures on {L}ie groups}.
\newblock University of Chicago Press, Chicago, IL, 1982.
\newblock Midway Reprint of the 1969 original.

\bibitem{ja:shagh}
J.~F. Adams.
\newblock {\em Stable homotopy and generalised homology}.
\newblock Chicago Lectures in Mathematics. University of Chicago Press,
  Chicago, IL, 1995.
\newblock Reprint of the 1974 original.

\bibitem{gb:itctg}
Glen~E. Bredon.
\newblock {\em Introduction to compact transformation groups}.
\newblock Academic Press, New York, 1972.
\newblock Pure and Applied Mathematics, Vol. 46.

\bibitem{mc:tuoecb}
Michael Cole, J.~P.~C. Greenlees, and I.~Kriz.
\newblock The universality of equivariant complex bordism.
\newblock {\em Math. Z.}, 239(3):455--475, 2002.

\bibitem{CMW}
S.~R. Costenoble, J.~P. May, and S.~Waner.
\newblock Equivariant orientation theory.
\newblock {\em Homology Homotopy Appl.}, 3(2):265--339 (electronic), 2001.
\newblock Equivariant stable homotopy theory and related areas (Stanford, CA,
  2000).

\bibitem{ae:rmaaisht}
A.~D. Elmendorf, I.~Kriz, M.~A. Mandell, and J.~P. May.
\newblock {\em Rings, modules, and algebras in stable homotopy theory},
  volume~47 of {\em Mathematical Surveys and Monographs}.
\newblock American Mathematical Society, Providence, RI, 1997.
\newblock With an appendix by M. Cole.

\bibitem{mg:atafc}
Marvin~J. Greenberg and John~R. Harper.
\newblock {\em Algebraic topology: a first course}, volume~58 of {\em
  Mathematics Lecture Note Series}.
\newblock Addison-Wesley publishing company, Reading, Mass., 1981.
\newblock Revised.

\bibitem{mj:hcfekt}
Michael Joachim.
\newblock Higher coherences for equivariant {$K$}-theory.
\newblock In {\em Structured ring spectra}, volume 315 of {\em London Math.
  Soc. Lecture Note Ser.}, pages 87--114. Cambridge Univ. Press, Cambridge,
  2004.

\bibitem{LM}
L.~G. Lewis, Jr. and Michael~A. Mandell.
\newblock Equivariant universal coefficient and {K}\"unneth spectral sequences.
\newblock {\em Proc. London Math. Soc. (3)}, 92(2):505--544, 2006.

\bibitem{ll:esht}
L.~G. Lewis, Jr., J.~P. May, M.~Steinberger, and J.~E. McClure.
\newblock {\em Equivariant stable homotopy theory}, volume 1213 of {\em Lecture
  Notes in Mathematics}.
\newblock Springer-Verlag, Berlin, 1986.
\newblock With contributions by J. E. McClure.

\bibitem{jm:ehactdttmorjp}
J.~P. May.
\newblock {\em Equivariant homotopy and cohomology theory}, volume~91 of {\em
  CBMS Regional Conference Series in Mathematics}.
\newblock Published for the Conference Board of the Mathematical Sciences,
  Washington, DC, 1996.
\newblock With contributions by M. Cole, G. Comeza\~{n}a, S. Costenoble, A. D.
  Elmendorf, J. P. C. Greenlees, L. G. Lewis, Jr., R. J. Piacenza, G.
  Triantafillou, and S. Waner.

\bibitem{gs:ekt}
Graeme Segal.
\newblock Equivariant {$K$}-theory.
\newblock {\em Inst. Hautes \'Etudes Sci. Publ. Math.}, (34):129--151, 1968.


\bibitem{gw:pdiektfc}
G.R. Williams.
\newblock {\em Poincar\'{e} duality in equivariant $K$-theory for $\C P(V)$}.
\newblock PhD thesis, University of Sheffield, 2005.

\end{thebibliography}
\end{document}